\documentclass{article}
\usepackage[T2A]{fontenc}
\usepackage[cp1251]{inputenc}
\usepackage[english]{babel}
\usepackage{color}
\usepackage{amssymb,amsmath,amsthm,amsfonts}
\newtheorem{teo}{Theorem}[section]

\begin{document}

\addcontentsline{toc}{section}{Abdullaev R.Z, B.A.~Madaminov
Isomorphisms and isometries of $F$-spaces of $log$-integrable
measurable functions}

\begin{center}
\bf Isomorphisms and isometries of $F$-spaces of $log$-integrable measurable functions

R.Z.~Abdullaev, B.A.~Madaminov

arustambay@yandex.com;  \ aabekzod@mail.ru
\end{center}


\begin{abstract}
In the paper, it is given isomorphic classification of $F$-spaces
of $log$-integrable measurable functions constructed using
different measure spaces. At the same time, it is proved that such
spaces are non-isometric.
\end{abstract}

\section{Introduction} \
One of the important classes of Banach functional spaces are
spaces $L_{p}(\Omega, \mathcal A, \mu),$ $1\leq p< \infty$ of all
$p$-th power integrable measurable functions given on the
measurable space $(\Omega, \mathcal A, \mu)$ with the finite
measure $\mu$ (almost everywhere equal functions are identified).
The study of isometries of Banach spaces $L_{p}$ was initiated by
S. Banach \cite{B}, who gave a description of all isometries for
spaces $L_p[0,1],$ $p\neq 2$. In \cite{L}, J. Lamperti gave
characterization of all linear isometries for $L_p$-spaces
$L_p(\Omega,\mathcal A,\mu)$ where $(\Omega, \mathcal A,\mu)$ is
an arbitrary space with the finite measure $\mu.$ The final result
in this setting is due to Yeadon \cite{Y} who gave a complete
description of all isometries between $L_p$-spaces associated with
different measure spaces. One of the corollaries of such
descriptions of isometries in spaces $L_{p}(\Omega, \mathcal A,
\mu)$ is the establishment of isometry for $L_{p}$-spaces
$L_{p}(\Omega, \mathcal A, \mu)$  and $L_{p}(\Omega, \mathcal A,
\nu)$ in the case when the measures $\mu$  and  $\nu$  are
equivalent.

An important metrizable analogue of Banach spaces $L_{p}$ are
$F$-spaces $L_{\log}(\Omega, \mathcal A,\mu)$ of $log$-integrable
measurable functions introduced in the work \cite{DSZ}. An
$F$-space $L_{\log}(\Omega, \mathcal A,\mu)$ is defined by the
equality
$$
L_{\log}(\Omega, \mathcal A, \mu)=\{f \in L_0(\Omega, \mathcal A,
\mu): \int \limits_{\Omega} \log(1+|f|)d\mu< + \infty\}
$$
where $ L_0(\Omega, \mathcal A, \mu$ is the algebra of all
measurable functions given on $(\Omega, \mathcal A, \mu)$ (almost
everywhere equal functions are identified).  By virtue of the
inequality
$$log(|f(\omega)|)\leq\frac{1}{p}|f(\omega)|^p, \ \omega \in
\Omega, \ p\in [1,\infty),$$ the inclusion $L_{p}(\Omega, \mathcal
A, \mu)\subset L_{log}(\Omega, \mathcal A, \mu)$ is always true.

In \cite{DSZ}, it was established that $L_{log}(\Omega, \mathcal
A, \mu)$ is a subalgebra in the algebra $L_0(\Omega, \mathcal A,
\mu).$ In addition, a special $F$-metric $\rho(f,g)$ has been
introduced in $L_{log}(\Omega, \mathcal A, \mu),$
$$
\rho(f,g)=\int_{\Omega} \log(1+|f-g|)d\mu, \ f, g \in
L_{\log}(\Omega, \mathcal A, \mu).
$$
The pair $(L_{\log}(\Omega, \mathcal A, \mu), \rho)$ is a complete
metric topological vector space with respect to this measure, and
the operation of multiplication $f\cdot g$ is continuous in the
totality of variables.

It is evident, algebras $ L_0(\Omega, \mathcal A, \mu)$ and $
L_0(\Omega, \mathcal A, \nu)$ coincide when the measures  $\mu$
and $\nu$ are equivalent. This fact is no longer true for the
algebras $L_{log}(\Omega, \mathcal A, \mu)$ and $L_{log}(\Omega,
\mathcal A, \nu).$ In the work \cite{AC}, it was shown that
$L_{log}(\Omega, \mathcal A, \mu)= L_{log}(\Omega, \mathcal A,
\nu)$ for equivalent measures $\mu$ and $\nu$ if and only if
$\frac{d\nu}{d\mu} \in L_{\infty}(\Omega, \mathcal A, \mu)$  and
$\frac{d\mu}{d\nu} \in L_{\infty}(\Omega, \mathcal A, \nu),$ where
$ L_{\infty}(\Omega, \mathcal A, \mu)$ is the algebra of all
essentially bounded measurable functions given on $(\Omega,
\mathcal A, \mu)$ (almost everywhere equal functions are
identified), $\frac{d\nu}{d\mu}$ (respectively,
$\frac{d\mu}{d\nu}$) is the Radon-Nikodym derivation of the
measure $ \nu $ (respectively, $\mu$) with respect to the measure
$\mu$ (respectively,  $\nu$).

Isometries on these $F$-spaces were considered in [5], [6]. In
these papers, a description of isometries on $F$-spaces was given.
In contrast to these results, in this paper we establish a
necessary and sufficient condition for the existence of isometries
and isomorphisms of algebras of log-integrable functions
constructed by different measures $\mu$ and $\nu.$ The
relationship between these isometries and isomorphisms is also
studied. In this case, conditions are imposed only on the
Radon-Nikodym derivatives $\frac{d\nu}{d\mu}.$

Naturally the problem arises to find the necessary and sufficient
conditions providing an isomorphism of the algebras $L_{log}
(\Omega, \mathcal A, \mu)$ and $L_{log}(\Omega, \mathcal A, \nu)$
for the equivalent measures $\mu$ and $\nu.$ The solution of this
problem is given in section 4 (see also \cite{AC}).

The main purpose of this paper is to prove the absence of
surjective isometries from $L_{log}(\Omega, \mathcal A, \mu)$ onto
$L_{log}(\Omega, \mathcal A, \nu)$ in the case when measures $\mu$
and $\nu$ are equivalent (see section 4).

In the last fifth section, we consider the $F$-space
$$L^{(\nu)}_{\log}(\Omega, \mathcal A, \mu)=\left\{f \in L_{0}(\Omega,
\mathcal A, \mu)):
\int\limits_{\Omega}\log(1+\frac{d\nu}{d\mu}\cdot |f|)d\mu < +
\infty\right\}$$ for equivalent measures $\mu$ and $\nu$ with the
$F$-norm
$$\|f\|^{(\nu)}_{log,\mu}=\int\limits_{\Omega}\log(1+\frac{d\nu}{d\mu}\cdot
|f|)d\mu.$$ It is proved that this \ $F$-space \
$L^{(\nu)}_{\log}(\Omega, \mathcal A, \mu)$ \ is a subalgebra \ in
\ $ L_0(\Omega, \mathcal A, \mu$) \ if \ and \ only \ if \
$(\frac{d\nu}{d\mu})^{-1} \in L_{log}(\Omega, \mathcal A, \mu)$.

\section{Preliminaries}

Let $(\Omega, \mathcal A, \mu)$ be a be a complete measure space
with finite measure $\mu$, and let  $L_0(\Omega, \mathcal A, \mu)$
(respectively, $ L_{\infty}(\Omega, \mathcal A, \mu)$) be the
algebra  of equivalence classes of real valued measurable
functions (respectively, essentially bounded real valued
measurable functions)  on $(\Omega, \mathcal A, \mu)$. Denote by
$\nabla=\nabla_\mu$ the complete Boolean algebra of all
equivalence classes $e=[A],$ $A \in \mathcal A,$ of equal
$\mu$-almost everywhere sets from the $\sigma$-algebra $\mathcal
A$. It is known that $\widehat{\mu}(e) =
\widehat{\mu}([A])=\mu(A)$ is a strictly positive finite measure
on $\nabla_\mu.$ In what follows, we also denote the measure
$\widehat {\mu}$ by $\mu,$ and the algebra $L_0(\Omega, \mathcal
A, \mu)$ (respectively, $ L_{\infty}(\Omega, \mathcal A, \mu)$) by
$L_0(\nabla)=L_0(\nabla_\mu)$ (respectively, $
L_{\infty}(\nabla)=L_{\infty}(\nabla_\mu)).$

Following to \cite{DSZ}, consider in $L_0(\nabla_\mu)$ a
subalgebra
$$
L_{\log}(\nabla_\mu)=\{f \in L_0(\nabla_\mu): \int
\limits_{\Omega} \log(1+|f|)d\mu< + \infty\} $$ of
$log$-integrable measurable functions, and for each $f \in
L_{\log}(\nabla_\mu),$ set
$$
\|f\|_{\log}=\int_{\Omega} \log(1+|f|)d\mu.
$$
By \cite[Lemma 2.1]{DSZ}, a nonnegative function
$\|\cdot\|_{\log}: L_{\log}(\nabla_\mu) \rightarrow [0,\infty)$
is a $F$-norm, that is,

$(i)$. $\|f\|_{\log}>0$ for all $0 \neq f \in L_{\log}(\nabla_\mu);$

$(ii)$. $\|\alpha f\|_{\log}\le\|f\|_{\log}$  for all $f \in
L_{\log}(\nabla_\mu)$ and real number $\alpha$ with $|\alpha|\leq
1;$

$(iii)$. $\lim_{\alpha\to 0}\|\alpha f\|_{\log}=0$ for all $f \in
L_{\log}(\nabla_\mu);$

$(iv)$. $\|f+g\|_{\log}\leq\|f\|_{\log}+\|g\|_{\log}$ for all $f,
g \in L_{\log}(\nabla_\mu).$

In \cite{DSZ} it is shown that $L_{\log}(\nabla_\mu)$  is a
complete topological algebra with respect to the topology
generated by the metric  $\rho(f,g)=\|f-g\|_{\log}.$

Let $\mu$ and $\nu$ be two equivalent finite measures on the
measurable space $(\Omega, \mathcal A)$ (writing $\mu \sim \nu$).
Since $\mu(A) =0$ if and only if $\nu(A) =0, \ A \in \mathcal A$,
it follows that
$$
L_{0}(\nabla_\mu) = L_{0}(\nabla_\nu):=L_{0}(\nabla); \ \ L_{\infty}(\nabla_\mu) = L_{\infty}(\nabla_\nu) :=L_{\infty}(\nabla).
$$
Let $\frac{d\nu}{d\mu}$ be the Radon-Nikodym  derivative of
measure $\nu$ with respect to the measure $\mu$. It is well known
that $0 \leq \frac{d\nu}{d\mu} \in L_{0}(\nabla)$ and
$$
f \in  L_1(\Omega, \mathcal A, \nu) \ \Longleftrightarrow \ f\cdot
\frac{d\nu}{d\mu} \in L_1(\Omega, \mathcal A, \mu),
$$
in addition,
$$\int\limits_{\Omega}\ f \ d\nu =\int\limits_{\Omega}\ f \cdot(\frac{d\nu}{d\mu}) \
d\mu.
$$
Note that for $\mu \sim \nu,$ it follows that
$(\frac{d\nu}{d\mu})^{-1} = \frac{d\mu}{d\nu}.$

\section{Isometries of the $F$-spaces \ $L_{\log}(\nabla_\mu)$ \ and \ $L_{\log}(\nabla_\nu)$}

 In this section, a necessary and sufficient condition for the existence of isometries is established of $L_{log}(\Omega, \mathcal
A, \mu)$ onto $L_{log}(\Omega, \mathcal A, \nu).$

Let $\mu\sim \nu$, \ $h=\frac{d\nu}{d\mu}$, and let
$$L_p(\nabla_\mu)=\{f \in L_{0}(\nabla_\mu): \ \|f\|_{p}=
(\int\limits_{\Omega}|f|^p d\mu)^\frac{1}{p} <  \infty\};$$
$$L_p(\nabla_ \nu)=\{f \in L_{0}(\nabla_\nu): \ \|f\|_{p}=(\int\limits_{\Omega}|f|^p
d\nu)^\frac{1}{p} =(\int\limits_{\Omega}h\cdot|f|^p
d\mu)^\frac{1}{p}<  \infty\}.
$$
If $\mu\neq \nu$ then $\mu(\{h \neq \mathbf 1\} > 0$, where
$\mathbf1(\omega)= \mathbf 1$ ($\mu$-almost everywhere). In this
case the map $U: L_p(\nabla_ \mu) \to L_p(\nabla_ \nu)$, defined
by the following equality
$$U(f)=h^{-\frac{1}{p}}f, \ f\in L_p(\nabla_\mu),$$
is the non trivial surjective isometry from $L_p(\nabla_ \mu)$
onto $L_p(\nabla_\nu)$.

Bellow we show that this statement is not true for $F$-spaces
$L_{\log}(\nabla_\mu)$ and $L_{\log}(\nabla_\nu)$.

Note that
$$L_{\log}(\nabla_\nu)= \{f \in L_{0}(\nabla): \int\limits_{\Omega} \log(1+|f|) \ d\nu \ < + \infty \}=$$
$$=\{f \in L_{0}(\nabla): \int\limits_{\Omega} h\cdot\log((1+|f|) \ d\mu < + \infty\}=$$
$$=\{f \in L_{0}(\nabla): \int\limits_{\Omega}\log((1+|f|)^h) \ d\mu < + \infty\},$$
and
$$\|f\|_{log,\nu}= \int\limits_{\Omega} log(1+|f|) \ d\nu =\int\limits_{\Omega}\log((1+|f|)^h) \ d\mu. $$

Let $\nabla$ be a non-atomic complete Boolean algebra, that is, a
Boolean algebra $ \nabla$ has not atoms. Let $\nabla_e=\{g \in
\nabla:g \leq e\},$ where $ 0\neq e \in \nabla$. By
$\tau(\nabla_e)$ denote the minimal cardinality of a set that is
dense in $ \nabla_e$ with respect to the order topology
($(o)$-topology). The non-atomic complete Boolean algebra $
\nabla$ is said to be \emph{homogeneous} if
$\tau(\nabla_e)=\tau(\nabla_g)$ for any nonzero $ e, g \in
\nabla.$ The cardinality $\tau(\nabla)$ is called the
\emph{weight} of the homogeneous Boolean algebra $\nabla$ (see,
for example \cite[chapter VII]{V}).

\begin{teo}\label{t32}
Let $\nabla$ be a complete homogeneous Boolean algebra, $\mu$ and $\nu$ be
finite equivalent measures on $\nabla.$
$L_{log}(\nabla_\mu)$ is isometric to $L_{log}(\nabla_\nu)$ iff
 $\frac{\int\limits_\Omega h d\mu}{\mu(\Omega)}=1.$
\end{teo}

\begin{proof}
Necessity is proved in [5], see Theorem 3.1. We prove sufficiency.
Let $\frac{\int \limits_{\Omega}h d\mu}{\mu(\Omega)}=1$. Then
using the equality $\int \limits_{\Omega}h d\mu=\int
\limits_{\Omega}h^{-1}h d\nu=\nu(\Omega),$ where
$h^{-1}=\frac{d\mu}{d\nu}$, we obtain $\nu(\Omega)=\mu(\Omega)$.
Hence, there is a measure-preserving automorphism $\alpha$ from
$\nabla_ \mu$ onto $\nabla_ \nu$, i.e. $\mu(e)=\nu(\alpha(e))$ for
any $e\in\nabla_ \mu$ ([6],VII.2. Theorem 5). Denote by $J_{\alpha}$ the isomorphism  of the algebra $L_0(\nabla_{\mu})=L_0(\nabla_{\nu})$ such that $J_{\alpha}(e) = \alpha(e)$ for all $e \in \nabla_{\mu}$. That's why for any $
f\in L_{log}(\nabla_\mu),$ we have from ([5], Proposition 3)
$$
\|f\|_{log,\mu}= \int\limits_{\Omega} log(1+|f(\omega)|)d\mu
=\int\limits_{\Omega} J_{\alpha}(log(1+|f(\omega)|))d\nu=
$$ $$
=\int\limits_{\Omega} (log(1+|J_{\alpha} f(\omega)|))d\nu=\|J_{\alpha}
f(\omega)\|_{log,\nu}.
$$
Hence, the  $ J_{\alpha}$  is a bijective linear isometry from
$L_{log}(\nabla_ \mu)$ onto $L_{log}(\nabla_ \nu)$.
\end{proof}

Let $h=\frac{d\nu}{d\mu}$ and $\Omega_>=\{\omega\in \Omega: h(\omega)>1\}, \
\Omega_<=\{\omega\in \Omega: h(\omega)<1\},
 \\[2mm]
 \Omega_==\{\omega\in \Omega: h(\omega)=1\}=\Omega\setminus(\Omega_>\cup\Omega_<). $
 Denote:
 \\[2mm]
 $S_>=S_>(\Omega,h)=$ $\frac{\int \limits_{\Omega_>}(h(x)-1)d\mu}{\mu(\Omega_>)}$ and\ $S_<=S_<(\Omega,h)=$ $\frac{\int \limits_{\Omega_<}(1-h(x))d\mu}{\mu(\Omega_<)}.$
The following theorem establishes 5 conditions equivalent to the isometricity of F-spaces.

\begin{teo}
Let $\nabla$  be a complete homogeneous algebra, $\mu$ and $\nu$
be finite equivalent measures on $\nabla,$
then the following conditions are equivalent\\
$(i).$ \ \  $ L_{log}(\nabla_\mu)$ and $ L_{log}(\nabla_\nu)$ are
isometric;\\[2mm]
$(ii).\ \  \frac{\int \limits_{\Omega}h(\omega)d\mu}{\mu(\Omega)}=1;$
\\[2mm]
$(iii). \ \  \frac{\int
\limits_{\Omega}h^{-1}(\omega)d\nu}{\nu(\Omega)}=1,$ where
$h^{-1}=\frac{d\mu}{d\nu};$ \\[2mm]
$(iv). \ \  \nu(\Omega)=\mu(\Omega);$ \\
$(v). \ \ S_>=S_<$;\\[2mm]
$(vi).$ \ \  there is a measure-preserving automorphism $\alpha$ from
$\nabla_ \mu$ onto $\nabla_ \nu.$

\end{teo}
\begin{proof}
$(i)\Longleftrightarrow(ii)$ follows from Theorem 3.1.
\\[2mm]
$(ii)\Longleftrightarrow(iii)$  $\frac{\int
\limits_{\Omega}h^{-1}(\omega)d\nu}{\nu(\Omega)}=\frac{\int
\limits_{\Omega}h(\omega)h^{-1}(\omega)d\mu}{{\int
\limits_{\Omega}d\nu}}=\frac{\mu(\Omega)}{{{\int
\limits_{\Omega}h(\omega)d\mu}}}=1.$ The reverse implication is proved
similarly.
\\[2mm]
$(iii)\Longleftrightarrow(iv)$ $\frac{\int
\limits_{\Omega}h^{-1}(\omega)d\nu}{\nu(\Omega)}=\frac{\int
\limits_{\Omega}h(\omega)h^{-1}(\omega)d\mu}{{\int
\limits_{\Omega}d\nu}}=\frac{\mu(\Omega)}{\nu(\Omega)}=1\Longleftrightarrow\mu(\Omega)=\nu(\Omega)$
\\[2mm]
$(iv)\Longleftrightarrow(vi)$ follows \cite[chapter VII, \S \ 2,
Theorems 5 and 6]{V})
\\[2mm]
$(ii)\Longleftrightarrow(v)$   $0=\frac{\int \limits_{\Omega}h(\omega)d\mu}{\mu(\Omega)}-1=\frac{\int \limits_{\Omega}h(\omega)d\mu}{\mu(\Omega)}-\frac{\int \limits_{\Omega}d\mu}{\mu(\Omega)}=\frac{\int \limits_{\Omega}(h(\omega)-1)d\mu}{\mu(\Omega)}=\frac{\int \limits_{\Omega_>}(h(\omega)-1)d\mu}{\mu(\Omega_>)}+\frac{\int \limits_{\Omega_<}(h(\omega)-1)d\mu}{\mu(\Omega_<)}+\frac{\int \limits_{\Omega_=}(h(\omega)-1)d\mu}{\mu(\Omega_=)}=S_>-S_<\Longleftrightarrow S_>=S_<$
\end{proof}

\section{The connection between isomorphisms and isometries of algebras $L_{\log}(\nabla_\mu)$ and $L_{\log}(\nabla_\nu)$}

Let $(\Omega, \mathcal A, \mu)$ and $(\Omega, A, \nu)$ be measurable spaces with
finite measures $\mu$. Let  $\nabla_{\mu}$ be a
complete Boolean algebra of all equivalence classes of equal
$\mu$-almost everywhere sets from the $\sigma$-algebra $\mathcal
A,$ and let $L_{\log}(\nabla_{\mu})$ be an algebra of
$log$-integrable measurable functions, corresponding measurable
space $(\Omega, \mathcal A, \mu_)$. Clear that $\nabla_{\mu}
\subset L_{\log}(\nabla_{\mu})$, and the restriction $\alpha =
\Phi |_{\nabla_{\mu}}$ of an isomorphism $\Phi:
L_0(\nabla_{\mu}) \to L_0(\nabla_{\mu})$ on a Boolean algebra
$\nabla_{\mu}$ is an isomorphism from a Boolean algebra
$\nabla_{\mu}$ onto a Boolean algebra $\nabla_{\mu}.$

The function $m(e)=\mu\circ\alpha(e)$, $e\in\nabla_\nu$ is the strictly  positive  finite measure  on  Boolean algebra  $\nabla_\nu$ and
$$L_{\log}(\nabla_\nu,\nu)=\Phi(L_{\log}(\nabla_\mu, \mu)) = L_{\log}(\nabla_m, m)$$ (see \cite[Proposition 3]{AC}).
Consequently, we get the following
Theorems (see \cite[Theorem 2]{AC}):
\begin{teo} $[7]$
$L_{\log}(M,\mu)=L_{\log}(M, \nu)$ if and only if $\frac{d\nu}{d\mu}, \frac{d\mu}{d\nu}. \in
L_\propto(\nabla_\mu)$
\end{teo}

\begin{teo}\label{t41}
$[7]$The algebras $ L_{\log}(\nabla_\mu, \mu)$ and $L_{\log}(\nabla_\nu,
\nu)$ are $\ast$-iso\-mor\-phic if and only if
$\frac{dm}{d\nu} \in  L_{\infty}(\nabla_\nu)$   and
$\frac{d\nu}{dm} \in L_{\infty}(\nabla_\nu)$.
\end{teo}

Let $\mu$ be a strictly positive finite measure on $ \nabla$. In
this case $\nabla$ is  a  direct product of homogeneous Boolean
algebras $ \nabla_{e_n},$ $e_n\cdot e_m = 0,$ $n\neq m,$
$\tau_n=\tau (\nabla_{e_n})< \tau_{n+1},$ $n,m \in \mathbb N$
\cite[chapter VII, \S 2, Theorem 5]{V}), where $\mathbb N$ is a
set of natural numbers. Let \ \ $\{\mu_1, \mu_2, \mu_3,...\}$ sequence  values of measures on units of homogeneous components on Boolean
algebra $\nabla,$  i. e. $\mu(e_n)=\mu_n, \ \ \forall n\in N.$

Let $\nabla$ be a complete Boolean algebra, $\mu$ and $\nu$ be
finite equivalent measures on $\nabla.$ Let $\nabla_i$ be
homogeneous components $\nabla,$ and let $ \Omega_i$ be a measurable set corresponding $e_i$, $h_i (x)=h(x)\mathfrak{X}_i$
where $\mathfrak{X}_i$ is the characteristic function of the set
$\Omega_i.$
The following theorem generalizes Theorem 3.1 for homogeneous Boolean algebras.
\begin{teo}\label{t42}
If $\mu$ and $\nu$ are
finite equivalent measures on $\nabla,$ then
$L_{log}(\nabla_\mu)$ is isometric to $L_{log}(\nabla_\nu)$ iff

 $\frac{\int
\limits_{\Omega_i}h(x)d\mu}{\mu(\Omega_i)}= 1$ for some $i\in
N.$
\end{teo}
\begin{proof}
Let $U$ be an isometry from $L_{log}(\nabla_\mu)$ onto
$L_{log}(\nabla_\nu)$. Then by virtue of Corollary 3.2 [8]
$U/\nabla_\mu$ transfers homogeneous components  $\nabla_\mu$ into
homogeneous ones. Then, by virtue of Theorems 3.1, the equality
$\frac{\int \limits_{\Omega_i } h(x) d\mu}{\mu(\Omega_i)}=1$
holds. Conversely, let $\frac{\int \limits_{\Omega_i } h(x)
d\mu}{\mu(\Omega_i)}=1.$ Then using the equality $\int
\limits_{\Omega_i}h(x)d\mu=\int
\limits_{\Omega_i}h^{-1}(x)h(x)d\nu=\nu(\Omega_i),$ where
$h^{-1}=\frac{d\mu}{d\nu}$, we obtain
$\nu(\Omega_i)=\mu(\Omega_i)$. Hence, there is a
measure-preserving automorphism $\alpha$ from $\nabla_ \mu$ onto
$\nabla_ \nu$, i.e. $\mu(x)=\nu(\alpha(x))$ for any $x\in\nabla_
\mu$ ([6],VII.2. Theorem 5). Let $J_{\alpha}$ be the isomorphism  of the algebra $L_0(\nabla_{\mu})=L_0(\nabla_{\nu})$ such that $J_{\alpha}(e) = \alpha(e)$ for all $e \in \nabla_{\mu}$.
. That's why for any $ f\in
L_{log}(\nabla_\mu),$ we have from ([5], Proposition 3)
$$
\|f\|_{log,\mu}= \int\limits_{\Omega} log(1+|f(x)|)d\mu
=\int\limits_{\Omega} J_{\alpha}(log(1+|f(x)|))d\nu=
$$ $$
=\int\limits_{\Omega} (log(1+|J_{\alpha} f(x)|))d\nu=\|J_{\alpha}
f(x)\|_{log,\nu}.
$$
Hence, the  $ J_{\alpha}$  is a bijective linear isometry from
$L_{log}(\nabla_ \mu)$ onto $L_{log}(\nabla_ \nu)$.
\end{proof}

Let $\{\mu_i\}$ and $\{\nu_i\}$ be sequences of the values of the
measures $\mu$ and $\nu,$ respectively on $\Omega_i,$ moreover
$\nabla_\mu=\nabla_\nu.$ We need the following theorem which gives
a classification of Boolean algebras with probability measures
\cite[chapter VII, \S \ 2, Theorems 5 and 6]{V}),\ \ for the case
when $\nabla_\mu=\nabla_\nu$.

\begin{teo} \cite[chapter VII, \S \ 2,
Theorems 5 and 6]{V}) The coincidence of sequences $\{\mu_i\}$ and
$\{\nu_i\}$ is necessary and sufficient for the existence of a
measure-preserving isomorphism from $\nabla_\mu$ onto
$\nabla_\nu.$
\end{teo}
We obtain from Theorem 4.3 the following
\begin{teo}   The $F$-spaces $L_{log}(\nabla_\mu)$ and  $L_{\log}(\nabla_\nu)$ are isometric if and only if \ $\mu_n  =\nu_n $  for all $n=1,2\dots$
\end{teo}

\begin{teo}
Let $\nabla$  be a complete algebra, $\mu$ and $\nu$
be finite equivalent measures on $\nabla,$
then the following conditions are equivalent\\
$(i).$ \ \  $ L_{log}(\nabla_\mu)$ and $ L_{log}(\nabla_\nu)$ are
isometric;\\[2mm]
$(ii).\ \  \frac{\int \limits_{\Omega_i}h(\omega)d\mu}{\mu(\Omega_i)}=1;$
\\[2mm]
$(iii). \ \  \frac{\int
\limits_{\Omega_i}h^{-1}(\omega)d\nu}{\nu(\Omega_i)}=1,$ where
$h^{-1}=\frac{d\mu}{d\nu};$ \\[2mm]
$(iv). \ \  \nu(\Omega_i)=\mu(\Omega_i);$ \\
$(v).$ \ \  there is a measure-preserving automorphism $\alpha$ from
$\nabla_ \mu$ onto $\nabla_ \nu.$
$(vi). \ \ S_>(\Omega_i,h)=S_<(\Omega_i,h)$;

\end{teo}
Let $\nabla$ be a complete homogeneous Boolean algebra, $\mu$ and $\nu$ are equivalent measures on $\nabla$, $h=\frac{d\nu}{d\mu}.$  We introduce the following notation

- by $C_{1}(h)$ we denote the fulfillment of the conditions $\frac{\int \limits_{\Omega_i} h(x) d\mu}{\mu(\Omega_i)}=1.$

-by $C_{2}(h)$ we denote the fulfillment of the conditions $h=\frac{d\nu}{d\mu}$ and $h^{-1}=\frac{d\mu}{d\nu}$ are bounded.

Consider the equivalent measures $\mu$ and $\nu$ on the Boolean
algebra $\nabla_\mu=\nabla_\nu.$ Let $\alpha$ be an automorphism
from $\nabla_\mu$ onto $\nabla_{\mu\cdot\alpha^{-1}}=\nabla_\nu,$
where ${\mu\cdot\alpha^{-1}}=m$ is a strictly positive measure on
$\nabla_{\mu\cdot\alpha^{-1}}.$ Then the equalities
$\nu(x)=m(\frac{d\nu}{d\mu}\cdot x)= \mu(\frac{dm}{d\mu}
\cdot\frac{d\nu}{dm}\cdot x)$ hold, and we obtain from Theorems
4.6 and 4.2 the following

\begin{teo}

(i) Algebras $L_{log}(\nabla_\mu)$ and $L_{log}(\nabla_\nu)$ are
isometric iff the condition $C_1(\frac{dm}{d\mu})$ holds.

(ii) Algebras $L_{log}(\nabla_\mu)$ and $L_{log}(\nabla_\nu)$ are
isomorphic iff the following conditions hold:
$C_1(\frac{dm}{d\mu})$ and $C_2(\frac{d\nu}{dm}).$
\end{teo}

Further, we obtain from Theorem 1 and ([4], Theorem 2) the
following

\begin{teo}

$(i)$ $L_{\log}(\nabla_\mu)$ and $L_{\log}(\nabla_m)$ are isometric and $L_{\log}(\nabla_\mu)$ and $L_{\log}(\nabla_m)$ are coincide if and only if the conditions $C_1(\frac{dm}{d\mu})$ and $C_2(\frac{dm}{d\mu})$ are satisfied.

$(ii)$ $L_{\log}(\nabla_\mu)$  and $L_{\log}(\nabla_m)$ are isometric and $L_{\log}(\nabla_\mu)$  and $L_{\log}(\nabla_m)$ are non coincide if and only if the condition $C_1(\frac{dm}{d\mu})$ is satisfied, but $C_2(\frac{dm}{d\mu})$ is not satisfied.

$(iii)$ $L_{\log}(\nabla_\mu)$ and $L_{\log}(\nabla_m)$ are non isometric and $L_{\log}(\nabla_\mu)$ and $L_{\log}(\nabla_m)$ are coincide if and only if the condition $C_1(\frac{dm}{d\mu})$ is not satisfied, but $C_2(\frac{dm}{d\mu})$ is satisfied.

$(iv)$ $L_{\log}(\nabla_\mu)$ и $L_{\log}(\nabla_m)$  are non
isometric and $L_{\log}(\nabla_\mu)$ and $L_{\log}(\nabla_m)$ are
non coincide if and only if the conditions $C_1(\frac{dm}{d\mu})$
and $C_2(\frac{dm}{d\mu})$ are not satisfied.
\end{teo}

An example of the fulfillment of statement $(ii)$ for a
homogeneous Boolean algebra is the case when

\[ h = \begin{cases}
x^{-1/2}, & \text{ {at } $0\leq x\leq1/25$}\\
-25x/32+33/32, & \text{{at }     $1/25<x\leq1.$}
\end{cases} \]
In this case,  $h$ satisfies condition $C_1(h),$ but does not
satisfy condition $C_2(h)$ since $h$ is unlimited. Examples for
statements $(i),$ $(iii),$ $(iv)$ are constructed similarly.

\section{ Closedness criterion of $L_{log}^{(\nu)}(\nabla_\mu)$ with respect to multiplication}

Below we give necessary and sufficient conditions for the
existence of an isomorphism $\Phi: L_0(\nabla_{\mu}) \to
L_0(\nabla_{\nu})$ that use the internal structure of Boolean
algebras  $\nabla_{\mu},$.

Let $\mu$ and $\nu$ be equivalent measures, and
$h=\frac{d\mu}{d\nu}$ be the Radon-Nikodym derivative of the
measure $\nu$ with respect to the measure $\mu.$

Consider now the following analog of the $F$-space of
log-integrable measurable functions
$$
L^{(\nu)}_{\log}(\nabla_\mu)=\{f \in L_{0}(\nabla):
\int\limits_{\Omega}\log(1+h|f|)d\mu < + \infty\},
$$
with the $F$-norm
$\|f\|^{(\nu)}_{log,\mu}=\int\limits_{\Omega}\log(1+h|f|)d\mu.$\ \ \ \ (3)
\begin{teo}
The function $\|f\|^{(\nu)}_{log,\mu}$ satisfies the following
conditions:
\\
$(i)$. $\|f\|^{(\nu)}_{log,\mu}>0$ for all $0 \neq f \in L_{\log}(\nabla_\mu);$

$(ii)$. $\|\alpha f\|^{(\nu)}_{log,\mu}$  for all $f \in
L_{\log}(\nabla_\mu)$ and real number $\alpha$ with $|\alpha|\leq
1;$

$(iii)$. $\lim_{\alpha\to 0}\|\alpha f\|^{(\nu)}_{log,\mu}=0$ for all $f \in
L_{\log}(\nabla_\mu);$

$(iv)$. $\|f+g\|^{(\nu)}_{log,\mu}\leq\|f\|^{(\nu)}_{log,\mu}+\|g\|^{(\nu)}_{log,\mu}$ for all $f,
g \in L_{\log}(\nabla_\mu).$
\end{teo}

\begin{proof}
$$(i).\parallel f\parallel^{(\nu)}_{log,\mu}=0\Longrightarrow\int\limits_{\Omega}\log(1+h(x)|f(x)|)d\mu=0\Longrightarrow log(1+h(x)|f(x)|)=0\Longrightarrow$$
$$\Longrightarrow1+h(x)|f(x)|=1\Longrightarrow h(x)|f(x)|=0.$$
Since $S(h)=1,$ we have $f=0.$

If $f\in L^{\nu}_{log}(\nabla_{\mu}),$ then $hf\in
L_{log}(\nabla_{\mu}).$ Therefore statement $(ii)$ of this theorem
follows from Lemma $(ii)$ [4] and obvious relationships.
$$(ii).
\parallel \alpha f\parallel^{(\nu)}_{log,\mu}=\|\alpha h f\|_{\log,\mu}\leq \|h f\|_{\log,\mu}=\|f\|_{\log,\mu}^{(\nu)},$$ for each $\alpha$ with $|\alpha|\leq 1$.

$(iii)$ Using $(iii)$ from Lemma 2.1 [4], we obtain
$\lim\limits_{\alpha\rightarrow0}\parallel \alpha
f\parallel^{(\nu)}_{log,\mu}=\lim\limits_{\alpha\rightarrow0}\int\limits_{\Omega}\log(1+\alpha
h|f|)d\mu=\lim\limits_{\alpha\rightarrow0}\parallel \alpha
f\parallel_{log,\mu}=0$

$$(iv).\parallel f+g\parallel^{(\nu)}_{log,\mu}=\int\limits_{\Omega}\log(1+h(x)|f(x)+g(x)|)d\mu\leq\int\limits_{\Omega}(\log(1+h(x)|f(x)|)+$$
$$+log(1+h(x)|g(x)|))d\mu=\|f\|^{(\nu)}_{log,\mu}+\|g\|^{(\nu)}_{log,\mu}$$
\end{proof}

Unlike the case of $F$-spaces $L_{log}(\nabla_\mu)$ and
$L_{log}(\nabla_\nu)$, the pairs of $F$-spaces
$L_{log}(\nabla_\mu)$ and $L_{log}^{(\nu)}(\nabla_\mu)$ are
already isometric.
\begin{teo}
For equivalent and unequal measures $\mu$ and $\nu,$ the pairs of
$F$-spaces $L_{log}(\nabla_\mu)$ and $L^{(\nu)}_{log}(\nabla_\mu)$
are isometric.
\end{teo}
\begin{proof}
In fact, for equivalent measures $\mu$ and $\nu,$ the map $U:
L^{(\nu)}_{log}(\nabla_ \mu)\to L_{log}(\nabla_ \mu),$ defined by
the equality $U(f)=h^{-1}f,$ $f\in L^{(\nu)}_{log}(\nabla_ \mu),$
$h=\frac{d\nu}{d\mu},$ is a linear bijection from
$L_{log}^{(\nu)}(\nabla_\mu) $ onto $L_{log}(\nabla_\mu),$
moreover
$$
\|U(f)\|^{(\nu)}_{log,\mu} = \int\limits_{\Omega}\log(1+h|h^{-1}f|)d\mu=\int\limits_{\Omega}\log(1+|f|)d\mu=
\|f\|_{log,\mu}
$$
for all $f \in L^{(\nu)}_{log}(\nabla_ \mu).$
\end{proof}

\begin{teo}
The space $L_{\log}^{(\nu)}(\nabla_\mu )$ will be an algebra if
and only $h^{-1} \in L_{\log}(\nabla_\mu )$
\end{teo}

\begin{proof} Consider the map
$H:{{L}_{\log }}(\nabla_\mu )\to {{L}^{(\nu)}_{\log }}(\nabla_\mu
)$ given by the formula $H(x)={{h}^{-1}x}.$

Let ${{h}^{-1}}\in {{L}_{\log }}(\nabla_\mu ).$ Then we have for
$f,g\in L_{log}^{(\nu)}(\nabla_\mu):$
$$
f^{'}=hf\in L_{log}(\nabla_\mu),\ \ \ \ \ \  g^{'}=hg\in
L_{log}(\nabla_\mu). \eqno (4)
$$

By virtue of Lemma 2.3.(b)[4], for $f,g\in {{L}^{(\nu)}_{\log
}}(\nabla_\mu ),$ the following relations take place:
$$
\left\| fg \right\|^{(\nu)}_{\log,\mu }={{\left\| {{h}}fg
\right\|}_{\log,\mu }}={{\left\|
h{{{h}^{-1}}{f}^{'}}{h}^{-1}{{g}^{'}} \right\|}_{\log,\mu
}}={{\left\| {{{h}^{-1}}{f}^{'}}{{g}^{'}} \right\|}_{\log,\mu
}}\le
$$ $$
\le {{\left\| {{f}^{'}} \right\|}_{\log,\mu }}+{{\left\| {{g}^{'}}
\right\|}_{\log,\mu }}+{{\left\| {{h}^{-1}} \right\|}_{\log,\mu
}}.
$$

We get the inequality  $\left\| fg \right\|^{(\nu)}_{\log,\mu }\le
{{\left\| {{f}^{'}} \right\|}_{\log,\mu }}+{{\left\| {{g}^{'}}
\right\|}_{\log,\mu }}+{{\left\| {{h}^{-1}} \right\|}_{\log,\mu
}}.$ This inequality and relation (4) imply that the space
${{L}^{(\nu)}_{\log }}(\nabla_\mu )$ is closed with respect to
multiplication.

Let ${{h}^{-1}}\notin {{L}_{\log }}(\nabla_\mu ).$
Since $\mathbf{1}\in {{L}_{\log }}(\nabla_\mu ),$ we get
$f={{h}^{-1}}\mathbf{1}\in {{L}^{(\nu)}_{\log }}(\nabla_\mu ).$
Really, $\left\| f \right\|_{\log,\mu }^{(\nu) }=\left\|
{{h}^{-1}} \right\|_{\log,\mu }^{(\nu) }={{\left\| \mathbf{1}
\right\|}_{\log,\mu }}<\infty $. Consider now the element
${{f}^{2}}.$ We obtain from equalities
$$
\int\limits_{\Omega} (log(\mathbf{1}+\left| {{h}}{{f}^{2}}{} \right|))d\mu=\int\limits_{\Omega}
(log(\mathbf{1} +\left|h {{h}^{-1}}{{h}^{-1}} \right|))\mu=\int\limits_{\Omega}
(log(\mathbf{1} +\left| {{h}^{-1}} \right|))\mu=\infty
$$
that ${{f}^{2}}\notin {{L}^{(\nu)}_{\log }}(\nabla_\mu ).$
Therefore the space ${{L}^{(\nu)}_{\log }}(\nabla_\mu )$ is not an
algebra for ${{h}^{-1}}\notin {{L}_{\log }}(X,\mu ).$
\end{proof}

Unlike to ${{L}^{(\nu)}_{\log }}(\nabla_\mu ),$ the space
${{L}^{\nu}_{\log }}(\nabla_\mu )={{L}_{\log }}(\nabla_\nu )$ is
an algebra for any different $\mu$ and $\nu$ [4].

 If $h^{-1}\notin L_{log}(\nabla_\mu),$ then $L_{\log}^{(\nu)}(\nabla_ \mu)$ not algebra. So it makes no sense to talk an isomorphism of  $L_{\log}(\nabla_ \mu)$ and $L_{\log}^{(\nu)}(\nabla_ \mu)$, but they are
isometric.

R. Abdullaev

Tashkent University of Information Technologies,

Tashkent, 100200, Uzbekistan,

e-mail arustambay@yandex.com

B. Madaminov

Urgench state Unversity,

Urgench, 220100, Uzbekistan,

e-mail aabekzod@mail.ru

\end{document}